\documentclass[12pt,a4paper]{article}
\usepackage{amsfonts,amssymb,amsmath,graphicx}
\usepackage{theorem} 
\usepackage{bm}      
\usepackage{times}   
\usepackage{eucal}   
\usepackage[english]{babel}

\newcommand{\engschreiben}{\itemsep-0.05em plus 0.05em \parsep0em plus 0.05em}

  {\theorembodyfont{\rm}
  \newtheorem{defi}{Definition}[section]
  
  \newtheorem{exa}[defi]{Example}
  \newtheorem{exas}[defi]{Examples} }

  {\theorembodyfont{\it} 
  \newtheorem{lem}[defi]{Lemma}
  
  \newtheorem{thm}[defi]{Theorem}
  \newtheorem{cor}[defi]{Corollary}  }
{\theorembodyfont{\rm}\theoremstyle{change}\newtheorem{nrtxt}[defi]{\hspace*{-0.55ex}}
\newenvironment{txt}{\begin{trivlist}\item}{\end{trivlist}}

\newlength{\boxzeichenlaenge}
\settowidth{\boxzeichenlaenge}{~$Box$}
\newenvironment{proof}%
    {\begin{trivlist} \item {\it Proof.}}%
    {\hfill\rule{0ex}{0.1ex} \hphantom{~$\Box$}\hspace{-\boxzeichenlaenge}~\hfill\mbox{$\Box$}\end{trivlist}}

\let\phi=\varphi

\let\theta=\vartheta

\newcommand{\bF}{{\mathbb F}}

\newcommand{\cB}{{\mathcal B}}
\newcommand{\cC}{{\mathcal C}}
\newcommand{\cD}{{\mathcal D}}

\newcommand{\cF}{{\mathcal F}}

\newcommand{\cR}{{\mathcal R}}

\newcommand{\cV}{{\mathcal V}}

\DeclareMathOperator{\AG}{AG}

\DeclareMathOperator{\Aut}{Aut}
\DeclareMathOperator{\diag}{diag}

 \DeclareMathOperator{\GF}{GF}
\DeclareMathOperator{\GL}{GL}
\DeclareMathOperator{\GaL}{{\Gamma}L}

\DeclareMathOperator{\id}{id}

\DeclareMathOperator{\PG}{PG}

\newcommand{\rel}{\mathrel{\cR}}
\newcommand{\qrel}{\mathrel{\overline\cR}}
\newcommand{\card}{\mathbin{\#}}
\newcommand{\Matrixfeld}[4]{\left#1\!\begin{array}{*{#3}{c}}#4\end{array}\!\right#2}
\newcommand{\SMatrixfeld}[4]{\hbox{\scriptsize\arraycolsep=.5\arraycolsep
  $\left#1\!\begin{array}{*{#3}{c}}#4\end{array}\!\right#2$}}

\newcommand{\Mat}{\Matrixfeld()}
\newcommand{\SMat}{\SMatrixfeld()}

\newcommand{\tsk}{$t$-$(s,k,\lambda_t)$}

\begin{document}

\title{Lifting of divisible designs}
\author{Andrea Blunck \and Hans Havlicek \and Corrado Zanella}
\date{}

\maketitle

\centerline{\emph{Dedicated to Walter Benz on the occasion of his 75th
birthday}}

\begin{abstract}\noindent
The aim of this paper is to present a construction of $t$-divisible designs for
$t>3$, because such divisible designs seem to be missing in the literature. To
this end, tools such as finite projective spaces and their algebraic varieties
are employed. More precisely, in a first step an abstract construction, called
$t$-lifting, is developed. It starts from a set $X$ containing a $t$-divisible
design and a group $G$ acting on $X$. Then several explicit examples are given,
where $X$ is a subset of $\PG(n,q)$ and $G$ is a subgroup of $\GL_{n+1}(q)$. In
some cases $X$ is obtained from a cone with a Veronesean or an $h$-sphere as
its basis. In other examples $X$ arises from a projective embedding of a Witt
design. As a result, for any integer $t\geq 2$ infinitely many non-isomorphic
$t$-divisible designs are found.

\par~\par\noindent
2000 Mathematics Subject Classification. 05B30, 51E20, 20B25.

\par\noindent
Key words: divisible design, finite projective space, Veronese variety.
\end{abstract}

\section{Introduction}

\begin{nrtxt}
This paper is concerned with the construction $t$-divisible designs; see
Definition \ref{def:DD}. We shall frequently use the shorthand ``DD'' for
``divisible design''. A well known construction of a $t$-DD is due to
A.~G.~Spera \cite[Proposition~4.6]{spera-92a}. It uses a finite set $X$ of
points which is endowed with an equivalence relation $\rel$, a group $G$ acting
on $X$, and a subset $B$ of $X$ called the `base block'. Then, under certain
conditions, the action of $G$ on $X$ gives rise to a $t$-divisible design with
point set $X$, equivalence relation $\rel$, and the $G$-orbit of $B$ as set of
blocks. If all equivalence classes are singletons then Spera's construction
turns into a construction of $t$-designs due to D.~R.~Hughes \cite[Theorem
3.4]{hugh-64a}.
\par
C.~Cerroni, S.~Giese, R.~H.~Schulz, A.~G.~Spera, and others successfully made
use of Spera's construction and obtained examples of $2$- and $3$-DDs. See
\cite{cerr-02a}, \cite{cerr+sch-00a}, \cite{cerr+sch-01a}, \cite{cerr+sp-99},
\cite{giese-05a}, \cite{giese+h+s-05a}, \cite{schulz+s-98a},
\cite{schulz+s-98b}, \cite{spera-95}, and \cite{spera-00a}. We refer also to
\cite[3.1]{giese-05a} for a detailed survey. It seems, however, that no
examples of $t$-DDs for $t>3$ were constructed in this way.
\end{nrtxt}

\begin{nrtxt}
One of the results in the thesis of S.~Giese is a construction of a $2$-DD
which it is called ``Konstruktion (A)'' in \cite[p.~64]{giese-05a}: It starts
with a given $2$-DD, say $\cD$, a finite projective space $\PG(n+1,q)$ with a
distinguished hyperplane $H=\PG(n,q)$ and a distinguished point
$O\in\PG(n+1,q)\setminus H$, called the \emph{origin}. Assuming that the
dimension $n$ and the prime power $q$ are sufficiently large, the point set of
the given $2$-DD can be mapped bijectively onto a set of $n-1$-spaces of $H$
subject to certain technical properties. Then each of these subspaces is joined
with the origin. This gives an isomorphic copy of the given $2$-DD whose
``point set'' consists of hyperplanes of $\PG(n+1,q)$ through the origin. Then
a new $2$-DD, say $\cD'$, can be obtained from the action of the translation
group (with respect to $H$) on this model of the given $2$-DD. See
\cite[Satz~3.2.4]{giese-05a}. Consequently, the ``points'' of $\cD'$ are also
hyperplanes of $\PG(n+1,q)$, but not all through the origin. It turns out that
this construction can be repeated by embedding $\PG(n+1,q)$ as a hyperplane in
$\PG(n+2,q)$, choosing a new origin in $\PG(n+2,q)\setminus\PG(n+1,q)$, and so
on. In this way infinite series of $2$-DDs can be obtained from any given
$2$-DD.

Of course, there is also the possibility to start the construction of Giese
when $\cD$ is a $t$-DD ($t\geq 2$), since such a structure is also a $2$-DD. In
\cite[Lemma~3.2.18]{giese-05a} necessary and sufficient conditions are given
for $\cD'$ to be a $t$-DD. However, those conditions are in terms of the new
structure $\cD'$ rather than the initial structure $\cD$, whence they cannot be
checked at the very beginning.
\end{nrtxt}

\begin{nrtxt}
The aim of the present note is to present a construction of a $t$-DD which
generalizes the ideas from \cite{giese-05a}. We start with an abstract group
acting $G$ on some set $X$, and a $t$-DD embedded in $X$. Then, under certain
conditions which can be read off from Theorem \ref{thm:lifting}, a new $t$-DD
is obtained via the action of $G$ on $X$. This process will be called a
\emph{$t$-lifting}.

Several explicit examples for $t$-liftings are presented in Section
\ref{se:geometric}. We choose $X$ to be a cone (without its vertex) in a finite
projective space $\PG(n,q)$, and $G$ to be a certain group of matrices. This
approach is still very general, since there are many possibilities for $X$. In
particular, when the base of the cone is chosen to be a Veronese variety,
infinitely many non-isomorphic $t$-divisible designs can be found for any
$t\geq 2$; see Theorem \ref{thm:unendlich.viele}. The construction of Giese,
even after a finite number of iterations, is just a particular case of our
construction of a $2$-lifting in a finite projective space. However, in order
to get Giese's results in their original form, one has to adopt a dual point of
view. Cf.\ the remarks in \ref{:bemerkungen}.
\end{nrtxt}

\section{Construction of $t$-liftings}

\begin{nrtxt}
Assume that $X$ is a finite set of \emph{points}, endowed with an equivalence
relation ${\rel}$; its equivalence classes are called \emph{point classes}. A
subset $Y$ of $X$ is called \emph{$\rel$-transversal} if for each point class
$C$ we have $\#(C\cap Y)\leq 1$. Let us recall the following:
\end{nrtxt}

\begin{defi}\label{def:DD}
A triple $\mathcal{D}=(X,\cB,\rel)$ is called a \tsk-\emph{divisible design} if
there exist positive integers $t,s,k,\lambda_t$ such that the following axioms
hold:

\begin{itemize}
\engschreiben

\item[(A)] $\mathcal{B}$ is a set of $\rel$-transversal subsets of $X$, called \emph{blocks}, with
$\card B=k$ for all  $B \in\mathcal{B}$.

\item[(B)] Each point class has size $s$.

\item[(C)] For each $\rel$-transversal $t$-subset $Y\subset X$ there exist
exactly $\lambda_t$ blocks containing $Y$.

\item[(D)] $t\leq \frac{v}{s}$, where $v:=\card X$.
\end{itemize}
\end{defi}
\begin{txt}
Observe that (D) is necessary to avoid the trivial case where no
$\rel$-transversal $t$-subset exists.
\end{txt}

\begin{nrtxt}
Sometimes we shall speak of a $t$-DD without explicitly mentioning the
remaining \emph{parameters} $s$, $k$, and $\lambda_t$. According to our
definition, a block is merely a subset of $X$. Hence the DDs which we are going
to discuss are \emph{simple}, i.e., we do not take into account the possibility
of ``repeated blocks''. Cf.\ \cite[p.~2]{beth+j+l-99a} for that concept.

A divisible design with $s=1$ is called a \emph{design}; we refer to the two
volumes \cite{beth+j+l-99a} and \cite{beth+j+l-99b}. In design theory the
parameter $s$ is not taken into account, and a $t$-$(1,k,\lambda_t)$-DD with
$v$ points is often called a $t$-$(v,k,\lambda_t)$-design.
\end{nrtxt}

\begin{nrtxt}
One possibility to construct divisible designs is given by the following
theorem. The ingredients for this construction are a finite set $X$, a finite
group $G$ acting on $X$, and a so-called \emph{base divisible design}, say
$(\overline X,\overline\cB,\qrel)$. Its orbit under the action of $G$ will then
yield a DD. More precisely, we can show the following:
\end{nrtxt}

\begin{thm}[$t$-Lifting]\label{thm:lifting}
Let $X$ be a finite set, let $t$ be a fixed positive integer, let $(\overline
X,\overline \cB,\qrel )$, where $\overline X\subset X$, be a $t$-$(\overline
s,k,\overline\lambda_t)$-divisible design, and let $G$ be a group acting on
$X$. Suppose, furthermore, that the following properties hold:
\begin{enumerate}\engschreiben

    \item
For each $x\in X$ there is a unique element of\/ $\overline X$, say $\widehat
x$, such that $x^G=\widehat x^G$.

    \item
All orbits $\overline x{}^G$, where $\overline x\in \overline X$, have the same
cardinality.
    \item
Given any subset $Y=\{y_1,y_2,\ldots,y_t\}$ of $X$, for which $\widehat
Y:=\{\widehat y_1,\widehat y_2,\ldots,\widehat y_t\}$ is an $\qrel
$-transversal $t$-subset of\/ $\overline X$, there exists at least one $g\in G$
such that $Y^g=\widehat Y$.
    \item
All setwise stabilizers $G_{\overline Y}$, where $\overline Y\subset \overline
X$ is any $\qrel$-transversal $t$-subset, have the same cardinality.
    \item
All setwise stabilizers $G_{\overline B}$, where $\overline B\in \overline \cB$
is any block, have the same cardinality.
\end{enumerate}
Then $(X,\cB,\rel )$ with
\begin{equation}\label{eq:B+R}
    \cB:= \overline\cB{}^G=\{\overline B{}^g
    \mid
    \overline B\in\overline \cB, g\in G\},\;\;
    {\rel}:=\{(x,x')\in X\times X \mid (\widehat x, \widehat x')\in{\qrel} \},
\end{equation}
is a \tsk-divisible design, where
\begin{equation}\label{eq:DD-lambda}
  s=(\card\overline x{}^G)\overline s,\;\;\lambda_t:=\overline\lambda_t \,
  \frac{\card G_{\overline Y}}{\card G_{\overline B}}
\end{equation}
with arbitrary $\overline x$, $\overline Y$, and $\overline B$ as above.
\end{thm}

\begin{proof}
It is clear from (a) that $\rel$ is a well-defined equivalence relation. Due to
(a) and (b), all its equivalence classes have cardinality $(\card\overline
x{}^G)\overline s$, where $\overline x\in\overline X$ can be chosen
arbitrarily. This establishes the first equation in (\ref{eq:DD-lambda}).

Next, we show that
\begin{equation}\label{eq:schnitt}
  \forall\,\overline Z\subset\overline X,\; \forall\, g\in G,\mbox{ and }
  \forall\, \overline x\in \overline Z\cap\overline Z{}^g :
  \overline x{}^g=\overline x.
\end{equation}
To prove this assertion consider $\overline z:=\overline x{}^{g^{-1}}$. From
$\overline x\in\overline Z{}^g$ follows $\overline z\in\overline Z\subset
\overline X$, whence (a) yields $\overline z\in \overline x{}^G\cap\overline X
= \{\overline x\}$. Thus $\overline z=\overline x$ which of course means
$\overline x{}^g=\overline x$.

Now let $\overline Y$ be an $\qrel $-transversal $t$-subset of $\overline X$.
Denote by $\overline B$ one of the $\overline\lambda_t\geq 1$ blocks of the DD
$(\overline X,\overline \cB,\qrel )$ containing the point set $\overline Y$. We
claim that
\begin{equation}\label{eq:durch-Y}
\forall\,g\in G :
   \overline Y\subset \overline B{}^g
   \Leftrightarrow
   g\in G_{\overline Y}.
\end{equation}
If $\overline Y\subset \overline B{}^g$ then $\overline Y\subset \overline B
\cap \overline B{}^g$. We infer from (\ref{eq:schnitt}), applied to $\overline
B\subset\overline X$, that all elements of $\overline B\cap\overline B{}^g$
remain fixed under the action of $g$, whence $g\in G_{\overline Y}$; the
converse is trivial. Next we describe the stabilizer of the subset $\overline
B$ in the subgroup $G_{\overline Y}$. Taking into account that all our
stabilizers are in fact pointwise stabilizers we read off from $\overline
Y\subset\overline B$ that $G_{\overline B} \subset G_{\overline Y}$. This shows
\begin{equation}\label{eq:stabil-B-Y}
     G_{\overline Y}\cap G_{\overline B} = G_{\overline B}.
\end{equation}
By combining (\ref{eq:durch-Y}) with (\ref{eq:stabil-B-Y}) we see that the
orbit $\overline B{}^G$ contains precisely $({\card G_{\overline Y}}) / ({\card
G_{\overline B}})$ distinct subsets $\overline B{}^g$ passing through
$\overline Y$.
\par
If $\overline B{}'\neq \overline B$ is another block of $(\overline X,\overline
\cB,\qrel )$ through $\overline Y$ then, by $\card\overline B = \card \overline
B{}'$, there are elements $\overline x\in\overline B\setminus\overline B{}'$
and $\overline x{}'\in\overline B{}'\setminus\overline B$. As the $G$-orbits of
$\overline x$ and $\overline x{}'$ are disjoint due to (a), so are the
$G$-orbits of $\overline B$ and $\overline B{}'$. Consequently, the number of
blocks in $\cB$ containing $\overline Y$ equals the integer $\lambda_t$ as
defined in (\ref{eq:DD-lambda}).

Finally, let $Y=\{y_1,y_2,\ldots,y_t\}\subset X$ be any $\rel$-transversal
$t$-subset. Define the $t$-subset $\widehat Y\subset \overline X$ as in (c). By
the definition of $\rel$, this $\widehat Y$ is an $\qrel$-transversal
$t$-subset of $\overline X$. So there is a $g\in G$ with $Y^g=\widehat Y$.
Hence the number of blocks in $\cB$ containing $Y$ is $\lambda_t$, as
required.\end{proof}

\begin{txt}
We shall refer to the $t$-DD $(X,\cB,\rel)$ as a \emph{$t$-lifting\/} of the
$t$-DD $(\overline X,\overline\cB,\qrel)$ under the action of $G$. Clearly,
$v:=\card X=(\card x^G)\overline v $, where $\overline v:= \card\overline X$
and $x\in X$ can be chosen arbitrarily. Note that we did not exclude the case
$k=\overline v$ in the previous theorem. In this case the $t$-DD $(\overline
X,\overline \cB,\qrel)$ is trivial, since $\overline X$ is its only block, and
the lifted $t$-DD is transversal.

By construction, the group $G$ acts as a group of automorphisms of the $t$-DD
$(X,\cB,\rel)$. The group $G$ acts transitively on the set
of blocks if, and only if, the base DD has a unique block.

As has been noted, (\ref{eq:schnitt}) implies that for all sets $\overline
Z\subset \overline X$ the \emph{setwise\/} stabilizer $G_{\overline Z}$
coincides with the \emph{pointwise\/} stabilizer of $\overline Z$ in $G$. It is
therefore unambiguous to call $G_{\overline Z}$ just the \emph{stabilizer\/} of
$\overline Z$ in $G$, a terminology which is adopted below.

We recall from \cite{spera-92a} that a $t$-DD can be obtained with Spera's
construction if, and only if, it admits a group of automorphisms which acts
transitively on the set of blocks and transitively on the set of transversal
$t$-subsets of points. The following theorem states that under one additional
condition the procedure of $t$-lifting preserves the property that a $t$-DD can
be obtained with Spera's construction.
\end{txt}

\begin{thm}\label{thm:spera-invariant}
Let\/ $\cD=(X,\cB,\rel)$ be the $t$-lifting of a $t$-divisible design\/
$\overline\cD=(\overline X,\overline\cB,\qrel)$ under the action of $G$. Assume
that there is a group $\overline H$ of automorphisms of\/ $\overline\cD$ which
acts transitively on\/ $\overline\cB$ and transitively on the set of\/
$\qrel$-transversal $t$-subsets of $\overline X$. If each $\overline
h\in\overline H$ can be extended to an automorphism of\/ $\cD$, then\/ $\cD$
admits a group of automorphisms which acts transitively on $\cB$ and
transitively on the set of\/ $\rel$-transversal $t$-subsets of $X$. Hence\/
$\cD$ can also be obtained with the construction of Spera
\emph{\cite[Proposition~4.6]{spera-92a}}.
\end{thm}

\begin{proof}
Let $B_1,B_2\in\cB$ be blocks. So, by the definition of $\cB$, there exist
$g_1,g_2\in G$ and $\overline B_1, \overline B_2\in\overline\cB$ with
$B_i=\overline B{}_i^{g_i}$ for $i\in\{1,2\}$. The assumption on $\overline H$
gives the existence of an automorphism $h$ of $\cD$ such that $\overline
B{}_1^h= \overline B_2$. Hence $B_1^{g_1^{-1} h g_2}=B_2$, i.e., the
automorphism group of $\cD$ acts transitively on $\cB$.

The transitivity of the automorphism group of $\cD$ on the set of
$\rel$-transversal $t$-subsets of $X$ can be shown similarly.
\end{proof}

The following lemma gives a sufficient condition for an extension of an
automorphism of $\overline\cD$ to be an automorphism of $\cD$. We shall use it
in Theorem \ref{thm:spera-geht-auch}.

\begin{lem}\label{lem:normalisiert}
Let $\cD=(X,\cB,\rel)$ be the $t$-lifting of a $t$-divisible design $\overline\cD=(\overline X,\overline\cB,\qrel)$ under the action of $G$. Assume that an automorphism $\overline h$ of\/ $\overline\cD$ can be extended to a permutation $h$ of $X$ which normalizes the group of automorphisms of $\cD$ induced by $G$. Then $h$ is an automorphism of $\cD$.
\end{lem}

\begin{proof}
Since $h$ normalizes the automorphism group induced by $G$, the following
holds: For each $g\in G$ there exists $g'\in G$ with $x^{gh}=x^{hg'}$ for all
$x\in X$.
\par
Let $B\in\cB$ be a block. Hence $B=\overline B{}^g$ for some $g\in G$ and some
block $\overline B\in \overline\cB$. As $\overline B{}^{h}=\overline
B{}^{\overline h}$ is a block, so is $B^h= \overline B{}^{gh}=\overline
B{}^{hg'}$.
\par
Suppose that $C$ is a point class of $\cD$. Hence $C=\bigcup_{g\in G}\overline
C{}^g$ for some point class $\overline C$ of $\overline\cD$. Therefore
\begin{equation*}
    C^h = \bigcup_{g\in G}\overline C{}^{gh}
    = \bigcup_{g'\in G}\overline C{}^{hg'}
    = \bigcup_{g'\in G}\overline C{}^{\overline h g'}
\end{equation*}
is also a point class of $\cD$.
\end{proof}

The question arises, whether \emph{proper $t$-liftings\/} (i.e.\ $\overline
X\neq X$) do exist. The next theorem gives an answer.

\begin{thm}\label{thm:geht-immer}
Each $t$-divisible design $\overline\cD=(\overline X,\overline \cB,\qrel)$ can
be used as base for a proper $t$-lifting.
\end{thm}
\begin{proof}
We may assume that $\overline X=\{1,2,\ldots,\overline v\}$ is a set of
integers. We fix an integer $w\geq 1$ and write $W:=\{1,2,\ldots,w\}$. Let
$(G_i)_{i\in\overline X}$ be a family of subgroups (not necessarily distinct)
of the symmetric group of $W$. Assume, furthermore, that each $G_i$ acts
transitively on $W$. We now define $X:=\overline X\times W$, and then we
identify $i\in \overline X$ with the pair $(i,1)\in X$. Let $G$ be the direct
product $\prod_{i=1}^{\overline v} G_i$. An action of $G$ on $X$ is given by
defining the image of $(i,j)$ under $(g_1,g_2,\ldots,g_{\overline v})$ as
$(i,j^{g_i})$. Obviously, conditions (a), (b), and (c) in Theorem
\ref{thm:lifting} hold. Given an $\qrel$-transversal $u$-subset $\overline Z$
we obtain that $\card \overline Z{}^G=w^u$. Therefore
\begin{displaymath}
\card G_{\overline Z} = \frac{\card G}{w^u} ,
\end{displaymath}
whence also the remaining two conditions (d) and (e) are satisfied. So Theorem
\ref{thm:lifting} can be applied. For $w>1$ this yields a proper $t$-lifting.
\end{proof}

\begin{txt}
It should be noted that the lifted DD from the proof above allows an alternative description without referring to the group $G$: A subset of $X$ is a block if, and only if, its projection on $\overline X$ is a block of $\overline\cD$. The point classes of the lifted DD are the cartesian products of the point classes of $\overline\cD$ with $W$.

We shall present other, less trivial, general constructions for proper
$t$-liftings of an arbitrary $t$-DD in \ref{ex:liftbar}.
\end{txt}

\begin{nrtxt}
Let $s$ be a positive integer and $\cD = (X,\cB,\rel)$ a $t$-DD. Given $Y
\subset X$ denote by $Y^*$ the set of all $x\in X$ for which there exists an
$y\in Y$ with $x \rel y$. Then $\cD$ is called \emph{$s$-hypersimple} if for
every block $B$ and for every $\rel$-transversal $t$-subset $Y$ contained in
$B^*$ there exist exactly $s$ blocks $B_1,B_2,\ldots,B_s$ containing $Y$ and
such that $B^*_i$ = $B^*$ for each $i\in\{1, 2,\ldots,s\}$; see
\cite{spera-95}. The $t$-liftings described in Theorem \ref{thm:lifting} are
$s$-hypersimple with $s = \#G_Y /\#G_B$. It seems to be an open problem to find
regular $t$-divisible designs with $t> 3$ and which are not $s$-hypersimple for
any $s$.
\end{nrtxt}

\section{Geometric examples of $t$-divisible designs for any
$t$}\label{se:geometric}

\begin{txt}
In this chapter we focus our attention on $t$-DDs which arise from point sets
in a finite projective or affine space.

\end{txt}

\begin{thm}\label{thm:existenz-neu}
Let $t$ be a fixed positive integer and let\/ $\overline\cD=(\overline
X,\overline\cB,\overline{\rel})$ be a $t$-$(\overline s,k,\overline\lambda_t)$
divisible design with the following properties:
\begin{itemize}\engschreiben
\item[\emph{(i)}]

$\overline{X}$ is a set of $\overline v$ points generating a finite projective
space $\PG(d,q)$.

\item[\emph{(ii)}]

All\/ $\qrel$-transversal $t$-subsets of\/ $ \overline{X}$ are independent in
$\PG(d,q)$.

\item[\emph{(iii)}]

All blocks in\/ $\overline\cB$ generate subspaces of\/ $\PG(d,q)$ with the same
dimension $\beta-1$.

\end{itemize}
Then for each non-negative integer $c$ there exists a $t$-$(q^c\overline
s,k,q^{c(\beta-t)}\overline\lambda_t)$-divisible design with $q^c\overline v$
points.
\end{thm}

\begin{proof}
Let $c$ be a non-negative integer, $n:=d+c$, and identify $\PG(d,q)$ with the
subspace of $\PG(n,q)$ given by the linear system
\begin{equation*}\label{}
  x_{d+1}=x_{d+2}=\cdots = x_{n}=0.
\end{equation*}
Furthermore, choose $S\subset\PG(n,q)$ to be the $(c-1)$-dimensional subspace
\begin{equation*}\label{}
  x_{0}=x_{1}=\cdots = x_{d}=0.
\end{equation*}
Next, let $G$ be the multiplicative group formed by all upper triangular
matrices of the form
\begin{equation}\label{eq:matrix}
    \Mat2{I_{d+1} & M\\ 0 & I_{c}}\in\GL_{n+1}(q),
\end{equation}
where $M$ is any $(d+1)\times c$ matrix with entries in $\bF_q=\GF(q)$, $I_{*}$
stands for an identity matrix of the indicated size, and $0$ denotes a zero
matrix of the appropriate size. The group $G$ is elementary abelian, since it
is isomorphic to the additive group of $(d+1)\times c$ matrices over $\bF_q$.
By writing the coordinates of points as row vectors, the group $G$ acts in a
natural way (from the right hand side) on $\PG(n,q)$ as a group of projective
collineations. The subspace $S$ is fixed pointwise, and every subspace of
$\PG(n,q)$ containing $S$ remains invariant, as a set of points. We obtain
\begin{equation}\label{eq:bahn}
    \forall\,x\in \PG(n,q)\setminus S : x^G = (\{x\}\vee S) \setminus S,
\end{equation}
i.e., the orbit of a point $x$ not in $S$ is the $c$-dimensional affine space
which arises from the projective space $\{x\}\vee S$ by removing the subspace
$S$. We define $\pi:\PG(n,q)\setminus S \to \PG(d,q)$ to be the projection
through the centre $S$ onto $\PG(d,q)$. By (\ref{eq:bahn}), two points of
$\PG(n,q)\setminus S$ are in the same $G$-orbit if, and only if, their images
under $\pi$ coincide.

We shall frequently make use of the following \emph{auxiliary result}. Let $Q$
be an independent $(d+1)$-subset of $\PG(n,q)$ which together with $S$
generates $\PG(n,q)$. We claim that there is a unique matrix in $G$ taking each
element of $Q$ to its image under $\pi$. In order to show this assertion, we
choose a $(d+1)\times (d+1)$ matrix $L$ and a $(d+1) \times c$ matrix $M$ in
such a way that the rows of $(L \;M)$ represent the points of $Q$ (written in
some fixed order). Consequently, the rows of the matrix $(L\;0)$ represent the
$(d+1)$ points of $Q^\pi$ (ordered accordingly). By the exchange lemma, the
points of $Q^\pi$ are also independent, whence $L$ is invertible. We infer from
\begin{equation}\label{eq:transitiv}
\Mat2{L & M}\underbrace{\Mat2{I_{d+1} & -L^{-1}M \\ 0 & I_c}}_{{}:=g} = \Mat2{L
& 0}
\end{equation}
that $g\in G$ takes each point $x\in Q$ to $x^\pi\in Q^\pi$. Conversely, if a
matrix $\tilde g\in G$ takes $Q$ to $Q^\pi$ then $(L\;M)\cdot\tilde g=(L\;0)$,
so $\tilde g=g$.

\par
Finally, we define $X$ as the union of all orbits $\overline x^G$, where
$\overline x$ ranges in $\overline X$, and proceed by showing that the
assumptions (a)--(e) of Theorem \ref{thm:lifting} are satisfied:
\par
Ad (a): By (\ref{eq:bahn}), the projection $\pi$ maps each $x\in X$ to the only
element $\widehat x\in \overline X$ with the required property.
\par
Ad (b): All orbits $\overline x{}^G$, where $\overline x\in\overline X$, have
size $q^c$ according to (\ref{eq:bahn}).
\par
Ad (c): Let $Y$ be a subset of $X$, such that $\widehat Y$ is an
$\qrel$-transversal $t$-subset of $\overline X$. Due to our assumption (ii),
the projected $t$-subset $Y^\pi=\widehat Y$ of $\overline X$ is independent.
Thus it can be extended to a basis of $\PG(d,q)$ by adding a $(d-t+1)$-subset
$P$. The set $Y$ is independent because its projection is independent.
Moreover, $Q:=Y\cup P$ meets the requirement from our auxiliary result. Now the
matrix $g$ from (\ref{eq:transitiv}) takes $Y$ to $\widehat Y$.
\par
Ad (d): First, let $Y'\subset\PG(d,q)$ be the $t$-set of points given by the
first $t$ vectors of the canonical basis of $\bF_q^{d+1}$. So the pointwise
stabilizer of $Y'$ in $G$ consists of all matrices
\begin{equation}\label{eq:stabil_K}
    \Mat3{I_t & 0 & 0 \\ 0 & I_{d-t+1} & K \\ 0 & 0 & I_c},
\end{equation}
with an arbitrary $(d-t+1)\times c$ submatrix $K$ over $\bF_q$. Obviously, the
pointwise and the setwise stabilizers of $Y'$ in $G$ coincide.
\par
Next, suppose that $\overline Y\subset\overline X$ is an $\qrel$-transversal
$t$-subset, whence $\overline Y$ is independent. So $\overline Y$ can be
extended to a basis of $\PG(d,q)$. There exists a $(d+1)\times (n+1)$ matrix of
the form $(L\;0)$ whose rows represent the points of the chosen basis. Thereby
it can be assumed that the first $t$ rows are representatives for $\overline
Y$. We read off from
\begin{equation*}
    \Mat2{L^{-1} & 0\\ 0 & I_c}\Mat2{I_{d+1} & M \\ 0 & I_c}\Mat2{L & 0\\ 0 & I_c}
    = \Mat2{I_{d+1} & L^{-1} M\\ 0 & I_{c}},
\end{equation*}
where $M$ is arbitrary, that
\begin{equation*}
 G = \Mat2{L^{-1} & 0\\0 & I_c} G \Mat2{L & 0\\ 0 & I_c} \mbox{ ~and~ }
 G_{\overline Y} = \Mat2{L^{-1} & 0\\0 & I_c} G_{Y'} \Mat2{L & 0\\ 0 & I_c}.
\end{equation*}
Hence $\card G_{\overline Y}$ does not depend on the choice of $\overline Y$,
and (\ref{eq:stabil_K}) shows that
\begin{equation}\label{eq:lambda_t}
  \card G_{\overline Y}=q^{c(d-t+1)}.
\end{equation}

Ad (e): Choose any block $\overline B\in\overline\cB$. There exists an
independent $\beta$-subset $\overline Z\subset \overline B$. The setwise and
the pointwise stabilizers of $\overline Z$ and $\overline B$ in $G$ are all the
same. We may now proceed as in the proof of (d), with $t$, $Y'$, and $\overline
Y$ to be replaced by $\beta$, an adequate $\beta$-set $Z'$, and $\overline Z$,
respectively. Then (\ref{eq:lambda_t}) gives that
\begin{equation}\label{eq:bloecke_G}
  \card G_{\overline B}=q^{c(d-\beta+1)}
\end{equation}
has a constant value.

Now $\lambda_t=q^{c(\beta-t)}\overline\lambda_t$ is immediate from
(\ref{eq:DD-lambda}), (\ref{eq:lambda_t}), and (\ref{eq:bloecke_G}).
\end{proof}
\begin{txt}
Let us add some remarks on Theorem \ref{thm:existenz-neu}.
\end{txt}

\begin{nrtxt}\label{:bemerkungen}
The only reason for including condition (i) is to simplify matters. We could
also drop it and carry out our construction in the join of $S$ and the subspace
generated by $\overline X$.
\par
It is easily seen that the $t$-lifting process of Theorem
\ref{thm:existenz-neu} can be iterated. Given a base $t$-DD we may first apply
a $t$-lifting for some fixed integer $c_1>0$. This gives a second $t$-DD which
can be used as the base DD for a second $t$-lifting for some fixed integer
$c_2>0$. The $t$-DD obtained in this way may also be reached in a single step
from the initial base DD by applying a $t$-lifting with the integer
$c:=c_1+c_2$.
\par
Suppose that $t=2$, $c=1$. By removing the assumption (i), we obtain a
variation of Theorem \ref{thm:existenz-neu} which yields once more results from
\cite[Theorem~3.2.7]{giese-05a}. In order illustrate how the settings in
\cite{giese-05a} (hyperplanes of an affine space, translation group) correspond
to our settings, we merely have to adopt a dual point of view: Each point $p$
of $\PG(n,q)$ gives rise to the star of hyperplanes of $\PG(n,q)$ with vertex
$p$ or, said differently, a single hyperplane of $\PG(n,q)^*$. In this way we
obtain a bijective correspondence of $\PG(n,q)$ (as a set of points) with the
set of hyperplanes of its dual space $\PG(n,q)^*$. Due to $c=1$ the subspace
$S$ corresponds to a hyperplane of $\PG(n,q)^*$ which can be considered as
being at infinity. The group $G$ acts on the dual space as the corresponding
translation group. For an arbitrary $t$ and $c=1$ our Theorem improves
\cite[Proposition~3.2.9]{giese-05a}.
\par
There is a particular case, where we can give an alternative description of the
divisible design $(X,\cB,\rel)$ from Theorem \ref{thm:existenz-neu}.
\end{nrtxt}

\begin{cor}\label{cor:alternative}
Let $t$ be any positive integer and let $\overline{X}$ be a $k$-set of points
generating the projective space\/ $\PG(d,q)$, such that each $t$-subset of $
\overline{X}$ is independent, where $t\leq k$. We embed\/ $\PG(d,q)$ as a
subspace in\/ $\PG(n,q)$, where $n=d+c$ for some positive integer $c$, and
choose any subspace $S$ of\/ $\PG(n,q)$ complementary with $\PG(d,q)$. Define
$(X,\cB,\rel)$ as follows.
\begin{itemize}\engschreiben
\item[\emph{(i)}] $X$ is the cone with basis $\overline X$ and vertex $S$, but
without its vertex $S$.

\item[\emph{(ii)}] $\cB$ is the set of all sections $X\cap D$, where $D$ is
complementary with $S$.

\item[\emph{(iii)}] ${\rel}:=\{(x,x')\in X\times X \mid \{x\}\vee S =
\{x'\}\vee S \}$.
\end{itemize}
This $(X,\cB,\rel)$ is a transversal $t$-$(q^c,k,q^{c(d-t+1)})$-divisible
design.
\end{cor}
\begin{proof}
Let $ \overline{\cB}:=\{\overline{X}\}$ and let ${\qrel}$ be the diagonal
relation on $ \overline{X}$. The triple $(\overline{X},\overline{\cB},\qrel)$
is a trivial transversal $t$-$(1,k,1)$-DD with $ \overline{v}=k$ points and
just one block. Define $(X,\cB,\rel)$ as in the proof of Theorem
\ref{thm:existenz-neu}, where $\beta=d+1$. By (\ref{eq:bahn}), the point set
$X$ and the equivalence relation $\rel$ can be described as in (i) and (iii),
respectively. The auxiliary result in the proof of Theorem
\ref{thm:existenz-neu} shows that $G$ acts transitively on the set of
complements of $S$, whence (ii) characterizes the set of blocks.
\end{proof}

\begin{txt}
Next, we compare the lifting from the proof of Theorem \ref{thm:existenz-neu}
with Spera's construction.
\end{txt}

\begin{thm}\label{thm:spera-geht-auch}
Under the assumptions of Theorem \emph{\ref{thm:existenz-neu}} suppose that
there exists a group $\overline\Gamma$ of collineations of $\PG(d,q)$ which
acts on $\overline X$ as an automorphism group of the base $t$-DD
$\overline\cD$. Furthermore, we assume that $\overline\Gamma$ acts transitively
on the set $\overline\cB$ of blocks and transitively on the set of
$\qrel$-transversal $t$-subsets of $\overline X$. Then the $t$-lifting from the
proof of Theorem \emph{\ref{thm:existenz-neu}} yields $t$-divisible designs
which can also be obtained with Spera's construction
\emph{\cite[Proposition~4.6]{spera-92a}}.
\end{thm}
\begin{proof}
Let $\overline J\subset \GaL_{d+1}(q)$ be the group of those semilinear
bijections which give rise to collineations in $\overline\Gamma$. (In our
setting $\GaL_{d+1}(q) = \GL_{d+1}(q)\rtimes\Aut(\bF_q)$, i.e., a semilinear
transformation appears as a pair consisting of a regular matrix and an
automorphism of $\bF_q$.) Then
\begin{equation*}
  J:=\{(\diag(P,I_c),\zeta) \mid (P,\zeta)\in\overline J \}\subset\GaL_{n+1}(q)
\end{equation*}
is a group of semilinear transformations which yields a collineation group of
$\PG(n,q)$, say $\Gamma$. For each $\overline\gamma\in\overline\Gamma$ there is
at least one extension in $\Gamma$. Since $\overline X$ and $S$ remain
invariant under the collineations in $\Gamma$, so does the set $X$. A
straightforward computation shows that
\begin{equation}\label{eq:normalisiert}
   j^{-1}Gj=G \mbox{ for all }j\in J;
\end{equation}
here we identify each $g\in G$ with $(g,\id_{\bF_q})\in\GaL_{n+1}(q)$. We infer
from Lemma \ref{lem:normalisiert} that $\Gamma$ acts on $X$ as an automorphism
group of the lifted $t$-DD $\cD$. Thus Theorem \ref{thm:spera-invariant} can be
applied to the automorphism group of $\overline\cD$ given by $\overline\Gamma$.
Altogether, we obtain the required result: Spera's construction can be applied
to $X$, $\rel$, an arbitrarily chosen $\overline B\in\overline\cB$ as base
block, and the group $\langle G,J\rangle$ of semilinear transformations
generated by $G$ and $J$.
\end{proof}
\begin{txt}
If the collineation group $\overline\Gamma$ from the above has the additional
property to act transitively on the set of $\qrel$-transversal $t$-tuples of
$\overline X$ then $\langle G,J\rangle$ will even act transitively on the set
of $\rel$-transversal $t$-tuples of $X$. For, if $(y_1,y_2,\ldots,y_t)$ is such
a $t$-tuple then there is an element $g\in G$ taking $(y_1,y_2,\ldots,y_t)$ to
the $\qrel$-transversal $t$-tuple $(y_1^g,y_2^g,\ldots,y_t^g)$ according to
assumption (c) in Theorem \ref{thm:lifting}.
\end{txt}

\begin{exas}\label{:beispiele}
(a) The \emph{small Witt design\/} $W_{12}=(\overline X,\overline\cB,\qrel)$ is
a $5$-$(1,6,1)$-DD (i.e.\ a design) with $\overline v=12$ points. By a result
of H.~S.~M.~Coxeter \cite{cox-58}, $W_{12}$ can be embedded in $\PG(5,3)$ in
such a way that the following properties hold: (i) $\overline X$ generates
$\PG(5,3)$. (ii) All $5$-subsets of $\overline X$ are independent. (iii) All
blocks span hyperplanes of $\PG(5,3)$. In fact, the blocks are those $132$
hyperplane sections of $\overline X$ which contain more than three points of
$\overline X$. We refer to \cite{havl-99}, \cite{pell-74}, \cite{todd-59}, and
\cite{white-66} for further properties of this model of $W_{12}$.
\par
We can apply Theorem \ref{thm:existenz-neu} to construct $5$-$(3^c,6,1)$-DDs
with $12\cdot 3^c$ points from $W_{12}$.

By \cite{cox-58}, each automorphism of $W_{12}$ can be extended in a unique way
to a a collineation of $\PG(5,3)$ leaving invariant the set $\overline X$. The
automorphism group of $W_{12}$ is the Mathieu group $M_{12}$. So we have a
collineation group $\overline\Gamma$ which acts sharply $5$-transitively on
$\overline X$. Since each block is uniquely determined by five of its points,
all blocks are in one orbit of $\overline\Gamma$. By Theorem
\ref{thm:spera-geht-auch}, this implies that the lifted $5$-DDs could also be
obtained with the construction of Spera.
\par
(b) Let $\overline X$ be as in (a). Corollary \ref{cor:alternative}, applied to
the set $\overline X$, yields the existence of $5$-$(3^c,12,3^c)$-DDs with the
same set of points and the same point classes as in (a), but with a different
set of blocks. As before, the lifted DDs could also be obtained with the
construction of Spera.
\par
(c) The \emph{large Witt design\/} $W_{24}=(\overline X,\overline\cB,\qrel)$ is
a $5$-$(1,8,1)$-DD (i.e.\ a design) with $\overline v=24$ points and $758$
blocks. An embedding in $\PG(11,2)$ is due to J.~A.~Todd \cite{todd-59}. It has
the following properties: (i) $\overline X$ generates $\PG(11,2)$. (ii) All
$5$-subsets of $\overline X$ are independent. (iii) All blocks span
$6$-dimensional subspaces of $\PG(11,2)$. The automorphism group of $W_{24}$ is
the Mathieu group $M_{24}$ which acts $5$-transitively on the point set of
$W_{24}$. Each automorphism of $W_{24}$ extends to a unique collineation of
$\PG(11,2)$; see \cite{todd-59}. Mutatis mutandis, it is now possible to
proceed as in (a) and (b).
\par
(d) Any field extension $\bF_{q^h}/\bF_q$, $h>1$, gives rise to a \emph{chain
geometry\/} $\Sigma(\bF_q,\bF_{q^h})$; see, for example,
\cite[pp.~40--41]{blunck+h-05b} (``M{\"o}biusraum'') or \cite{herz-95}. Such a
chain geometry is a $3$-$(1,q+1,1)$-DD (i.e.\ a design) with $q^h+1$ points. We
speak of chains rather than blocks in this context. The following is due to
G.~Lunardon \cite[p.~307]{luna-84}: This design can be embedded in
$\PG(2^{h}-1,q)$ as an algebraic variety, say $\overline X$, called an
\emph{$h$-sphere\/}. Any three distinct points of $\overline X$ are
independent. Furthermore, all its chains span subspaces with a constant
dimension $\min\{q,h\}$. (The chains on the $h$-sphere are normal rational
curves; see \ref{:veronese} below.) Hence Theorem \ref{thm:existenz-neu} can be
applied to construct $3$-DDs from this embedded chain geometry. Observe that it
remains open from \cite{luna-84} whether or not $\overline X$ will always
generate $\PG(2^{h}-1,q)$.
\par
Each semilinear automorphism of this chain geometry extends to a collineation
of $\PG(2^{h}-1,q)$. The group of these collineations meets the conditions from
Theorem \ref{thm:spera-geht-auch}, whence one could also apply Spera's
construction to obtain the lifted $3$-DDs.
\par
We add in passing that for $h=2$ an $h$-sphere is just an elliptic quadric in
$\PG(3,q)$ and the associated design is a miquelian M{\"o}bius plane. Cf.\ also
\cite[pp.~48--50]{giese-05a}, where the case $h=2$, $c=1$, $q$ odd is treated
from a dual point of view.
\par
If we disregard the chains on the $h$-sphere then Corollary
\ref{cor:alternative} gives a $3$-DD with block size $q^h+1$.

\par
(e) Any generating set $\overline X$ of $\PG(d,q)$ yields a $2$-DD according to
Corollary \ref{cor:alternative}.
\end{exas}

\begin{nrtxt}\label{:veronese}
We proceed by showing that the assumptions of Corollary \ref{cor:alternative}
can be realized for each integer $t\geq 2$ if $\overline X$ is chosen as an
appropriate \emph{Veronese variety}.

Suppose that three integers $c,m\geq 1$, $t\geq 2$, and a finite field $\bF_q$
are given. We let $d={m+t-1\choose m}-1$ and consider the projective space
$\PG(d,q)$. Its $d+1$ coordinates will be indexed by the set $E_{m,t-1}$ of all
sequences $e=(e_0,e_1,\ldots,e_m)$ of non-negative integers satisfying
$e_0+e_1+ \cdots +e_m=t-1$; the coordinates are written in some fixed order.
The \emph{Veronese mapping\/} is given by
\begin{equation}\label{eq:veronese}
  v_{m,t-1}:\PG(m,q)\to \PG(d,q) : \bF_q(x_0,x_1,\ldots,x_m) \mapsto
  \bF_q(\ldots, y_{{e_0},{e_1},\ldots,{e_m}},\ldots),
\end{equation}
where $y_{{e_0},{e_1},\ldots,{e_m}}:=x_0^{e_0} x_1^{e_1} \cdots x_m^{e_m}$. Its
image is known as a \emph{Veronese variety\/} (or, for short a
\emph{Veronesean\/}) $\cV_{m,t-1}(q)$. A $\cV_{1,t-1}$ is also called a
\emph{normal rational curve}.
\par
There is a widespread literature on Veronese varieties. We refer to
\cite{herz-82} for a coordinate-free definition of the Veronese mapping which
allows to derive its essential properties in a very elegant way. See also
\cite{havl+z-97a}. The case of a finite ground field is presented in
\cite[Chapter~25]{hirs+t-91} for $t=3$, and in \cite{cossi+l+s-01} for
arbitrary $t$. Many references, in particular to the older literature (over the
real and complex numbers), can also be found in \cite{havl-03a}.

For the reader's convenience we present now two results together with their
short proofs. The first coincides with \cite[Corollary~2.6]{cossi+l+s-01}, the
second seems to be part of the folklore.
\end{nrtxt}

\begin{lem}\label{lem:vero-eigenschaften}
The following assertions hold:
\begin{enumerate}\engschreiben
\item\label{lem:spannt.auf}
The Veronesean $\cV_{m,t-1}(q)$ spans $\PG(d,q)$
if, and only if, $t \leq q+1$.

\item\label{lem:unabhaengig}
The Veronese mapping \emph{(\ref{eq:veronese})}
maps any $t\geq 2$ distinct points of\/ $\PG(m,q)$ to $t$ independent points
of\/ $\PG(d,q)$.
\end{enumerate}
\end{lem}

\begin{proof}
Ad (a): Each family $(a_e)_{e\in E_{m,t-1}}$ with entries in $\bF_q$, but not
all zero, corresponds in $\PG(d,q)$ to a hyperplane, say $H$, with equation
$\sum_{e\in E_{m,t-1}}a_ey_e=0$, and in $\PG(m,q)$ to an algebraic
hypersurface, say $\cF$, with degree $t-1$ which is given by
\begin{equation*}
    \sum_{e\in E_{m,t-1}} a_{{e_0},{e_1},\ldots,{e_m}}x_0^{e_0}x_1^{e_1}\cdots
    x_m^{e_m}=0.
\end{equation*}
A point $p$ of $\PG(m,q)$ is in $\cF$ if, and only if, its Veronese image is in
$H$. Clearly, all hyperplanes of $\PG(d,q)$ and all hypersurfaces with degree
$t-1$ of $\PG(m,q)$ arise in this way.

By a result of G.~Tallini \cite[p.~433--434]{tallini-61a} there are
hypersurfaces of any degree $\geq q+1$ containing \emph{all\/} points of
$\PG(m,q)$, but no such hypersurfaces of degree less than $q+1$. By the above,
this means that $\cV_{m,t-1}(q)$ does not span $\PG(d,q)$ precisely when
$t-1\geq q+1$.
\par
Ad (b): Let $p_1, p_2,\ldots,p_t$ be $t\geq 2$ distinct points of $\PG(m,q)$.
Choose one of them, say $p_t$. There exist (not necessarily distinct)
hyperplanes $Z_i$ of $\PG(m,q)$, such that $p_i\in Z_i$ and $p_t\notin Z_i$ for
all $i\in\{1,2,\ldots,t-1\}$. If $\sum_j c_{ij}x_j=0$ are equations for the
$Z_i$s then $\prod_{i=1}^{t-1}(\sum_j c_{ij}x_j)=0$ gives a hypersurface $\cF$
of degree $t-1$ which contains $p_1,p_2,\ldots p_{t-1}$, but not $p_t$. We
infer from the the proof of (\ref{lem:spannt.auf}) that there is a hyperplane
$H$ of $\PG(d,q)$ which contains the Veronese images of $p_1,p_2,\ldots
p_{t-1}$, but not the image of $p_t$. Thus the image of $p_t$ is not in the
span of the remaining image points.
\end{proof}

\begin{thm}\label{thm:unendlich.viele}
For any integer $t\geq 2$ there exist infinitely many non-isomorphic
transversal $t$-divisible designs.
\end{thm}
\begin{proof}
Fix any $t\geq 2$ and choose any integer $m\geq 1$. There is a prime power $q$
such that $t\leq q+1$. The Veronesean $\cV_{m,t-1}$ has
$k:=q^m+q^{m-1}+\cdots+1\geq q+1\geq t$ points, and it spans $\PG(d,q)$ by
Lemma~\ref{lem:vero-eigenschaften}~(\ref{lem:spannt.auf}). We read off from
Lemma~\ref{lem:vero-eigenschaften}~(\ref{lem:unabhaengig}) that any $t$ points
of $\cV_{m,t-1}=:\overline X$ are independent. So the assumptions of Corollary
\ref{cor:alternative} are satisfied. As $c$ runs in the set of non-negative
integers, we obtain infinitely many non-isomorphic transversal
$t$-$(q^c,k,q^{c(d-t+1)})$-DDs.
\end{proof}

\begin{txt}
Letting $m=c=1$ in the above proof yields a DD which is contained in a cone
with a one-point vertex over a normal rational curve $\cV_{1,t-1}$ in
$\PG(t-1,q)$. These DDs are finite analogues of \emph{tubular circle planes\/}
\cite[p.~398]{polster+s-01}. We refer also to \cite{cerr+sch-01a} (dual point
of view) and \cite{giese+h+s-05a} for the case when $m=c=1$ and $t=3$.

An alternative proof of Theorem \ref{thm:unendlich.viele} is provided by the
construction from Theorem \ref{thm:geht-immer}. One may start there with a
trivial $t$-DD with point set $\overline X:=\{1,2,\ldots,\overline v\}$,
$\overline\cB:=\{\overline X\}$, and the diagonal relation as $\qrel$. Then, as
$w$ varies in the set of non-negative integers, infinitely many non-isomorphic
$t$-DDs are obtained. However, this approach gives trivial $t$-DDs, because
\emph{every\/} $\rel$-transversal $\overline v$-subset of such a $t$-DD turns
out to be a block. The DDs which arise from the proof of
\ref{thm:unendlich.viele} are trivial if, and only if, the Veronesean
$\cV_{m,t-1}$ is a basis of $\PG(d,q)$, i.e.\ for $k=d+1$.

In the previous proof we could also choose $\overline X$ to be a
\emph{subset\/} of $\cV_{m-1,t}$ with at least $t$ elements. This would also
give a $t$-DD by applying the construction of Corollary \ref{cor:alternative}
to the subspace generated by $\overline X$. We confine our attention to one
particular case.
\end{txt}

\begin{exa}
In $\PG(d,q)$, i.e.\ the ambient space of the Veronesean $\cV_{m,t-1}$, let us
arrange the coordinates in such a way that the first $m+1$ coordinates belong
to the sequences
\begin{equation*}
  (t-1,0,0,\ldots 0), (t-2,1,0,\ldots 0), \ldots, (t-2,0,\ldots,0,1) \in E_{m,t-1}.
\end{equation*}
The order of the remaining coordinates is immaterial. As before, we embed
$\PG(m,q)$ via the Veronese mapping (\ref{eq:veronese}) in $\PG(d,q)$, and then
$\PG(d,q)$ in $\PG(n,q)$ via the canonical embedding (cf.\ the proof of Theorem
\ref{thm:existenz-neu}). Furthermore, we turn $\PG(m,q)$ into an affine space
by considering $x_0=0$ as its \emph{hyperplane at infinity}. The Veronese image
of an affine point $\bF_q(1,x_1,x_2,\ldots x_m)$ is
\begin{equation*}
   \bF_q(1,x_1,x_2,\ldots x_m, \underbrace{*,\ldots,*}_{d-m},
                               \underbrace{0,0,\ldots,0}_{c}).
\end{equation*}
Here the entries marked with an asterisk are polynomials in
$x_1,x_2,\ldots,x_m$. Let $\overline X$ be the set of all such points.

The minimum degree of a hypersurface in $\AG(m,q)$ containing \emph{all} points
of $\AG(m,q)$ is $q$. The proof is similar to the one for the projective case
\cite{tallini-61a}. So, provided that $t\leq q$, the set $\overline X$ spans
$\PG(d,q)$; see also Lemma~\ref{lem:vero-eigenschaften}~(\ref{lem:spannt.auf}).
Hence, for $t\leq q$ we obtain a $t$-$(q^c,q^m,q^{c(d-t+1)})$-DD by applying
Corollary \ref{cor:alternative}.

The action of $G$ on $X=\overline X^G$ is as follows: Any matrix
$g:=\SMat2{I_{d+1}& M\\ 0 & I_c}$ as in (\ref{eq:matrix}) takes
\begin{equation}\label{eq:G-vero-aff-ur}
  \bF_q(1,x_1,x_2,\ldots
  x_m,\underbrace{*,\ldots,*}_{d-m},y_1,y_2,\ldots,y_c),
\end{equation}
to
\begin{equation}\label{eq:G-vero-aff-bild}
  \bF_q(1,x_1,x_2,\ldots
  x_m,\underbrace{*,\ldots,*}_{d-m},y_1+P_1,y_2+P_2,\ldots,y_c+P_c),
\end{equation}
where each $P_j$, $j\in\{1,2,\ldots,c\}$, denotes a polynomial in
$x_1,x_2,\ldots,x_m$ with degree $\leq t-1$. The coefficients of $P_j$ are the
entries in the $j$th column of $M$.
\par

However, this DD admits an alternative description which avoids Veroneseans and
projective spaces.  We simply delete the block of $d-m$ coordinates and go over
to inhomogeneous coordinates in (\ref{eq:G-vero-aff-ur}) and
(\ref{eq:G-vero-aff-bild}). This amounts to applying a projection which maps
$X$ bijectively onto $\AG(m+c,q)$. We use this bijection to obtain an
isomorphic DD and an isomorphic action of the group $G$ on $\AG(m+c,q)$. It is
given by
\begin{equation*}
  (x_1,x_2,\ldots x_m,y_1,y_2,\ldots,y_c)\stackrel g\longmapsto
  (x_1,x_2,\ldots x_m,y_1+P_1,y_2+P_2,\ldots,y_c+P_c).
\end{equation*}
Hence the blocks of $\AG(m+c,q)$ are precisely the graphs of all the $c$-tuples
of polynomial functions $\bF_q^m\to \bF_q$ with degree $\leq t-1$, whereas the
point classes are the cosets of the subspace $x_1=x_2=\cdots= x_m=0$ in
$\bF_q^{m+c}$. In particular, when $m=c=1$ then the unique block through an
$\rel$-transversal $t$-subset of $\AG(2,q)$ is just the graph of the polynomial
function with degree $\leq t-1$ which is obtained by the interpolation formula
of Lagrange. Compare with \cite[p.~399--400]{polster+s-01} for similar results
over the real numbers. See also \cite{luksch-87} for a detailed investigation
of this ``geometry of polynomials''.

\end{exa}

\begin{exa}\label{ex:liftbar}
Let $(\overline X,\overline\cB,\qrel)$ be any $t$-DD with $\overline v$ points,
$t\geq 2$. There is a prime power $q$ such that $q+1\geq \overline v\geq t$. We
consider the normal rational curve $\cV_{1,t-1}$ in $\PG(t-1,q)$; it has $q+1$
points. So we can identify $\overline X$ with a subset of $\cV_{1,t-1}$. Now it
is easy to verify the conditions from Theorem \ref{thm:existenz-neu}, because
any $t$ distinct points of $\overline X$ form a basis of $\PG(t-1,q)$.

When $t=2$ then $\cV_{1,t-1}=\PG(1,q)$ is a projective line. In this particular
case the result can be found in \cite[Bemerkung 3.2.2]{giese-05a}.
\end{exa}

\begin{exa}
Let $\cC$ be a $[\nu,\kappa]$-linear code on $\mathbb{F}_q$ of minimum weight
$t+1\geq 3$. It is well known (cf. for example \cite{cecch+t-81a}) that $\cC$ is
associated with a $\nu$-set, say $\overline X$, of points in $\PG({\nu-\kappa-1},q)$,
such that every $t$-subset of $\overline X$ is independent and there exists a dependent
$(t+1)$-subset of $\overline X$. By Corollary \ref{cor:alternative}, for each $c\geq 1$
we obtain a transversal $t$-$(q^c,\nu,q^{c(\nu-\kappa-t)})$-DD.

On the other hand, each $t$-DD determines a constant weight code. See
\cite{schulz+s-00} and the references given there. Thus, according to our
construction, we can link two concepts from coding theory and it would be
interesting to know more about this connection.
\end{exa}

\begin{nrtxt}
In order to apply the construction of DDs according to Theorem
\ref{thm:existenz-neu} with an appropriate $t$ one could also embed a given DD
in an arc, an oval, a hyperoval, an ovoid, a cap of kind $t-1$ (any $t$ points
are independent), etc. Thus many more DDs can be constructed.
\par
The group $G$ used in the proof of Theorem \ref{thm:existenz-neu} is elementary
abelian and it yields a so-called \emph{dual translation group} of the lifted
DD. See \cite[Chapter~5]{giese-05a}, where characterizations of DDs admitting
such a group can also be found.
\par
Another promising setting for a $3$-lifting (according to Theorem
\ref{thm:lifting}) could be to use the projective line over a finite (not
necessarily commutative) local ring as $X$, and a suitable subgroup of the
general linear group $\GL_2(R)$ as $G$. Such a group need not be elementary
abelian. Here some overlap with the work of Spera \cite{spera-95}, who
considered the projective line over a finite local algebra and the full group
$\GL_2(R)$, is to be expected.
\end{nrtxt}

\footnotesize


Author's addresses:

\par
Andrea Blunck, Fachbereich Mathematik, Universit{\"a}t Hamburg, Bundesstra{\ss}e 55,
D-20146 Hamburg, Germany.\\ \mbox{\tt andrea.blunck@math.uni-hamburg.de}
\par
Hans Havlicek, Institut f{\"u}r Diskrete Mathematik und Geometrie, Technische
Universit{\"a}t Wien, Wiedner Hauptstra{\ss}e 8--10, A-1040 Wien, Austria.\\ \mbox{\tt
havlicek@geometrie.tuwien.ac.at}
\par
Corrado Zanella, Dipartimento di Tecnica e Gestione dei Sistemi Industriali,
Universit{\`a} di Padova, Stradella S. Nicola, 3, I-36100 Vicenza, Italy.\\
\mbox{\tt corrado.zanella@unipd.it}

\end{document}